\documentclass[12pt,reqno,oneside]{amsart}

\usepackage[utf8]{inputenc}
\usepackage[T1]{fontenc}

\usepackage{amsmath, amssymb, amsthm}
\usepackage{geometry}
\usepackage{setspace}
\usepackage{needspace}
\usepackage{enumitem}
\usepackage{hyperref}
\hypersetup{
  colorlinks,
  citecolor=blue,
  linkcolor=blue,
  urlcolor=blue}
\usepackage{marginnote}
\usepackage{xcolor}
\usepackage{upgreek}
\usepackage{tikz}
\usetikzlibrary{calc}
\usepgflibrary {shadings} 
\usetikzlibrary{intersections}
\usetikzlibrary{patterns}
\usepackage{mathtools}
\usepackage{verbatim}
\usepackage{lipsum}  
\usepackage{mathrsfs}
\usepackage{stmaryrd}
\usepackage[
    colorinlistoftodos,
    backgroundcolor=green!20!white,
    bordercolor=green!75!black,
    linecolor=green!75!black
]{todonotes}
\usepackage{aliascnt}
\usepackage{cleveref}
\usepackage{bm}

\numberwithin{equation}{section}

\theoremstyle{plain}
\newtheorem{theorem}{Theorem}[section]
\newtheorem*{theoremA}{Theorem A}
\newaliascnt{lemma}{theorem}
\newtheorem{lemma}[lemma]{Lemma}
\aliascntresetthe{lemma}
\newaliascnt{proposition}{theorem}
\newtheorem{proposition}[proposition]{Proposition}
\aliascntresetthe{proposition}
\newaliascnt{corollary}{theorem}

\aliascntresetthe{corollary}
\theoremstyle{definition}
\newaliascnt{definition}{theorem}

\aliascntresetthe{definition}
\theoremstyle{remark}
\newaliascnt{remark}{theorem}
\newtheorem{remark}[remark]{Remark}
\aliascntresetthe{remark}
\newtheorem*{remark*}{Remark}

\newcommand{\C}{\mathbf{C}}
\newcommand{\R}{\mathbf{R}}
\newcommand{\Q}{\mathbf{Q}}

\renewcommand{\div}{\operatorname{div}}

\renewcommand{\Re}{\operatorname{Re}}

\newcommand{\vol}{\mathrm{vol}}

\renewcommand{\b}[1]{\bm{#1}}

\newcommand{\dist}{\operatorname{dist}}
\newcommand{\Ric}{\operatorname{Ric}}
\newcommand{\Scal}{\operatorname{Scal}}
\newcommand{\spt}{\operatorname{spt}}
\newcommand{\tr}{\operatorname{tr}}

\newcommand{\eps}{\varepsilon}
\begin{document}

\title[Integral curvature estimates and geodesic limits]{Integral curvature estimates and geodesic limits for stable abelian Higgs fields}

\begin{abstract}
We prove that, on closed three-manifolds, the limit interface of a sequence of stable abelian Higgs critical points with uniformly bounded energy is a finite union of smooth closed immersed geodesics, possibly with multiplicity and self-intersections. This follows from an integral estimate for the diffuse curvature, derived from an integral bound on the discrepancy.
\end{abstract}

\author{Marco Badran}
\address{Department of Decision Sciences and BIDSA, Bocconi University, Milano, Italy}
\email{marco.badran@unibocconi.it}

\author{Marco A. M. Guaraco}
\address{
Department of Mathematics,
Imperial College London,\newline
Huxley Building,
180 Queen's Gate,
London SW7 2AZ,
United Kingdom
}
\email{guaraco@imperial.ac.uk}

\author{Aria Halavati}
\address{Department of Decision Sciences and BIDSA, Bocconi University, Milano, Italy}
\email{aria.halavati@unibocconi.it}

\maketitle

%\tableofcontents

\subsection*{AI usage statement.}
Inspired by the recent work of the first author on the classification of entire stable solutions, we felt it should be possible to classify limits of stable solutions in closed manifolds using geometric measure theory, without developing a strong regularity theory. We prompted the LLM Astra-6 to explore the possibilities in that direction. The LLM produced a first proof with little help from us. We checked and rewrote it in order to highlight its novel mathematical ideas and improve the clarity of the presentation.

\section{Introduction}

Diffuse approximations connect nonlinear elliptic equations with the geometry of minimal submanifolds. In codimension one, this connection goes back to the work of Modica and Mortola \cite{Modica-Mortola1977} on the Allen--Cahn equation.
In codimension two, the abelian Higgs model plays an analogous role. Let $(M,g)$ be a smooth oriented Riemannian manifold, let $\Omega\subset M$ be open, and let $L\to\Omega$ be a Hermitian line bundle. We consider the self-dual Yang--Mills--Higgs energy
\begin{equation}\label{eq: energy}
\begin{split}
E_\eps(u,\nabla;\Omega)&\coloneqq\frac12\int_\Omega e_\eps(u,\nabla)\,d\vol_g,\\
e_\eps(u,\nabla)&\coloneqq|\nabla u|^2+\eps^2|F_\nabla|^2+\frac1{4\eps^2}(1-|u|^2)^2,
\end{split}
\end{equation}
where $u$ is a section and $\nabla$ is a smooth unitary connection with real curvature $F_\nabla=i\nabla^2$. A stable critical point in $\Omega$ is defined as a pair $(u,\nabla)$ satisfying
\begin{equation*}
    \delta E_\eps(u,\nabla)=0,\quad \delta^2E_\eps(u,\nabla)\geq0,
\end{equation*}
where both variations are tested against fields 
compactly supported in $\Omega$.  The compactness theorem of Pigati and Stern \cite{Pigati-Stern2021} shows that, under a uniform energy bound, the energy measures of sequences of critical points with $\eps_k\to0$ converge subsequentially to the weight measure of a stationary integral varifold of codimension two, up to a fixed normalisation. This result paved the way to many recent advances in the existence of solutions, regularity and geometric evolution \cite{Badran-delPino2023,Badran-delPino2024,Parise-Pigati-Stern2024b,DePhilippis-Pigati2024,DePhilippis-Halavati-Pigati2026,Nguyen-Wang2026}.

\medskip

Stability has proved especially effective in codimension-one regularity theory \cite{Tonegawa2005,Tonegawa-Wickramasekera2012,Wickramasekera2014,Wang-Wei2019a,Wang-Wei2019b, Mantoulidis2021,Chodosh-Mantoulidis2020}, but its geometric consequences in codimension two remain less understood. Here we establish a regularity theorem for stable abelian Higgs fields in dimension three.

\begin{theoremA}
Let $\Omega$ be an open subset of a smooth oriented Riemannian three-manifold $(M^3,g)$. Consider sequences $\eps_k\to0$ and $(u_k,\nabla_k)$ of smooth stable critical points of \eqref{eq: energy} with $\eps=\eps_k$, $|u_k|\leq1$ and
\begin{equation*}
\sup_k E_{\eps_k}(u_k,\nabla_k;\Omega)\leq\Lambda.
\end{equation*}
Then, for every $\Omega'\Subset\Omega$,
\begin{enumerate}
\item The discrepancies
\begin{equation*}
\xi_k\coloneqq\eps_k|F_{\nabla_k}|-\frac{1-|u_k|^2}{2\eps_k}
\end{equation*}
satisfy $\|\eps_k^{-1}\xi_k\|_{L^1(\Omega')}\to0$.
\item Define $\nu_k\coloneqq(*F_{\nabla_k})^\sharp/|F_{\nabla_k}|$ where $F_{\nabla_k}\neq0$, and choose any measurable unit extension to its zero set. For the stress-energy tensor defined in \eqref{eq: stress convergence},
\begin{equation}
\lim_{k\to\infty}\|T_{\eps_k}-e_{\eps_k}\nu_k^\flat\otimes\nu_k^\flat\|_{L^1(\Omega')}=0.
\end{equation}
Moreover, the divergence-free vector field $q_{\eps_k}$ associated with the modified Jacobian in \eqref{eq: core vorticity formula} satisfies
\begin{equation}
\limsup_{k\to\infty}\int_{B_r(x)}|\nabla_{q_{\eps_k}}\nu_k|\,d\vol_g\leq C\sqrt r
\end{equation}
for every normal ball $B_{2r}(x)\subset\Omega'$ with $0<2r<r_0$. The integrand is extended by zero where $q_{\eps_k}=0$. The constants depend on $\Lambda$, the fixed interior domains and their geometry.
\item Every subsequential limit varifold supplied by \cite{Pigati-Stern2021} is locally a finite union of smooth geodesic arcs without interior endpoints, counted with positive integer multiplicities and possibly intersecting.
\end{enumerate}
If $\Omega=M$ is closed, each such limit is a finite union of closed immersed geodesics, counted with positive integer multiplicities.
\end{theoremA}

Theorem A, which follows from \Cref{thm: discrepancy}, \Cref{prop: integral curvature estimates,prop: stress jacobian approximation} and \Cref{thm: main} below,  concerns the geometry of the limiting varifold at arbitrary
multiplicity. It does not require a multiplicity-one hypothesis or exclude
intersections of distinct geodesics. Its proof has two main parts.

The analytic starting point is the the first author's work on entire stable abelian Higgs
fields~\cite{Badran2026b}. In that setting, appropriate energy growth and
expanding logarithmic cutoffs make the errors in a stability comparison
vanish, leading to exact identities and classification. We use the same
underlying test pairs with fixed compactly supported cutoffs. The cutoff
errors remain, and ambient curvature introduces additional terms, so the
comparison must be used quantitatively. 

The discrepancy $\xi_\eps$ compares the curvature and potential amplitudes.
The localized stability calculation, together with the critical-point
estimates of Pigati and Stern \cite{Pigati-Stern2021}, yields
\[
 \|d\xi_\eps\|_{L^2(\Omega')}
 +\|\eps^{-1}\xi_\eps\|_{L^2(\Omega')}\leq C,
 \qquad \Omega'\Subset\Omega.
\]
Passing from this bound to vanishing requires a further argument. The scalar
field equation identifies the weak limit of
$\eps^{-2}h_\eps\,d\vol_g$ with the energy measure. The nonnegative sequence
$\eps^{-1}\xi_\eps^-$ is bounded in $L^2$ and satisfies
$\eps^{-1}\xi_\eps^-\leq\eps^{-2}h_\eps$; any of its weak $L^2$ limits must
therefore be absolutely continuous and dominated by a measure concentrated
on a one-dimensional set. It follows that
$\eps^{-1}\xi_\eps\to0$ in $L^1_{\mathrm{loc}}$, after including the separate
positive-part estimate. Quantitative control thus replaces exact vanishing
at fixed $\eps$, and vanishing is recovered in the singular limit.

In three dimensions, the curvature also determines a vector field and its
unit direction,
\[
 b_\eps=(*F_{\nabla_\eps})^\sharp,
 \qquad
 \nu_\eps=\frac{b_\eps}{|b_\eps|}
 \quad\text{where }F_{\nabla_\eps}\neq0.
\]
Here the Hodge star and the metric identify two-forms with vectors. The same
stability calculation retains a second nonnegative term which, combined with
the discrepancy derivative bound and an identity for $b_\eps$, gives
\[
 \int_{\Omega'\cap\{|F_{\nabla_\eps}|>0\}}
 \eps^2|F_{\nabla_\eps}|^2|\nabla_{\nu_\eps}\nu_\eps|^2\,d\vol_g\leq C.
\]
This controls the bending of the magnetic direction with the curvature part
of the energy as weight. Further defect estimates follow by testing the
discrepancy equation against a cutoff times $h_\eps$. These deductions use
the discrepancy bounds already obtained, so the analytic estimates share the
same localized stability input.

To use these estimates in the limit, we use the modified Jacobian whose
associated vector field $q_\eps$ is divergence-free and supported in a fixed
sublevel set $\{|u_\eps|^2\leq b_0\}$, with $b_0<1$ (see \cite[III.25]{Bethuel-Brezis-Helein1993}). Although $b_\eps$ is
already divergence-free, localization near the vortex cores is needed to
compare the curvature weight with the energy density. A correction to the
cutoff magnetic field preserves zero divergence. The bending and defect
estimates then give the local bound
\[
 \limsup_{\eps\to0}\int_{B_r(x)}|\nabla_{q_\eps}\nu_\eps|\,d\vol_g
 \leq C\sqrt r.
\]
In particular, the contribution of this derivative becomes small on small
balls, uniformly after passage to the limit.

The use of Jacobians and their relation to the limiting geometry builds on
existing work. Pigati--Stern compactness includes identification of the
limiting stress tensor with the tangent projection of $V$. Jacobians
localized near vortex cores also appear in their work on the nonmagnetic
Ginzburg--Landau model~\cite{Pigati-Stern2023}. Related energy--Jacobian
compatibility near multiplicity-one planar limits is developed by De
Philippis, Halavati, and Pigati~\cite{DePhilippis-Halavati-Pigati2026}. We
verify the mass and direction compatibility needed for our choice of
$q_\eps$: $|q_\eps|\,d\vol_g\rightharpoonup2\pi\|V\|$, and the magnetic
direction, weighted by energy, converges to the unoriented tangent of $V$.
These are unsigned statements and are consistent with cancellation in an
oriented current limit. The additional control needed for the present
argument is the quantitative turning estimate and its use to exclude
branching.

The final step converts this estimate into a condition on each vertex. The
equation $\div q_\eps=0$ permits integration by parts with test functions
depending on both position and the direction $\nu_\eps$. The mass and
direction identification passes the position derivative to $V$, while the
turning estimate controls the term containing the direction derivative.
Testing near a vertex with a radial cutoff times an arbitrary smooth odd
function $a$ of the direction gives, first letting $\eps\to0$ and then
shrinking the ball,
\[
 \left|\sum_i m_i a(v_i)\right|\leq C_a\sqrt r,
 \qquad\text{hence}\qquad
 \sum_i m_i a(v_i)=0.
\]
The radial cutoff contributes a fixed endpoint term on each arm, whereas the
test function's derivative in the direction variable stays bounded as
$r\downarrow0$. Unlike stationarity, which tests only linear functions of the
direction, this identity holds for every odd $a$. It therefore implies
\[
 \sum_i m_i\delta_{v_i}=\sum_i m_i\delta_{-v_i}.
\]
Each arm pairs with an opposite arm of equal total multiplicity. Uniqueness
for the geodesic equation joins these pairs into full smooth geodesics,
proving the stated regularity of the limit.

%The discrepancy bound arises as a consequence of stability, using the Riemannian analogues of the test fields used in \cite{Badran2026b}. The discrepancy equation then gives vanishing of the first-order defect, which identifies the magnetic direction with the tangent direction of the limiting varifold. We also use the construction similar to \cite[(III.25)]{Bethuel-Brezis-Helein1993} of a divergence-free modified Jacobian supported in the vortex core.

%The quantitative discrepancy bounds and the gauge correction in the stability inequality yield integral curvature estimates for the magnetic direction. Combined with the linear energy growth, these show that its total turning along the modified Jacobian in a ball of radius $r$ is controlled, in the limit, by $C\sqrt r$. Testing the divergence-free identity against functions odd in the direction variable forces the outgoing directions at each vertex of the limiting geodesic network to pair antipodally, with matching multiplicities. The geodesic arcs therefore continue smoothly through every vertex, allowing intersections.

\begin{remark}
No better regularity than the one predicted by Theorem A can hold for limits of stable critical points, as shown by the following example. Consider the Riemannian manifold $(\mathbf{T}^3,g)$ with the metric $g=(1+\sin^2 y+V(z))dx+(1+\sin^2 x+V(z))dy+dz$, 
\begin{equation*}
    \quad V(z)=
\begin{cases}
e^{-1/z^2}\sin^2(1/z), & z\neq 0,\\
0, & z=0.
\end{cases}
\end{equation*}
consider the disjoint curves
$\gamma_{1,j}(s)=(s,0,a_j)$ and $\gamma_{2,j}(s)=(0,s,-a_j)$, where $a_j=(\pi j)^{-1}$.
Each pair is a strict local minimum of length, so the abelian Higgs $\Gamma$-limit \cite{Parise-Pigati-Stern2024a}
provides stable local minimizers concentrating on it as $\varepsilon\to0$.
A diagonal choice $\varepsilon_j\ll a_j$ yields a limiting pair of closed geodesics crossing transversely at the origin.
\end{remark}

\section{Preliminaries and the stability calculation}\label{sec: preliminaries}
We denote a typical section--real one-form pair by $\b\Phi=(\phi,\Phi)$, and reserve ${\bf U}=(u,\nabla)$ for the background solution. Throughout the remainder of the paper, $\dim M=3$ and $|u|\leq1$. Angled brackets denote the real fibre inner product, given in a unitary trivialisation by $\langle a,b\rangle=\Re(a\overline b)$. Locally we write $\nabla^A=d-iA$, where $A$ is a real one-form; then $F_\nabla=dA$. We also use $\nabla$ for the induced connections on tensor-valued sections and real tensors.
We use the convention that the Hodge Laplacian has nonpositive spectrum $\Delta\coloneqq-(dd^*+d^*d)$.  If $F$ is the skew-symmetric matrix representing $F_\nabla$ in an orthonormal frame, then
\begin{equation*}
|F_\nabla|^2=\frac12\sum_{i,j}F_{ij}^2,\qquad \|F\|^2=2|F_\nabla|^2,
\end{equation*}
where $\|\cdot\|$ denotes the Hilbert--Schmidt norm. 

\medskip

Set
\begin{equation*}
h(u)\coloneqq\frac12(1-|u|^2),\qquad j_\nabla u\coloneqq\langle\nabla u,iu\rangle.
\end{equation*}
When no confusion is possible, we write $h=h(u)$. The Euler--Lagrange equations of \eqref{eq: energy} are 
\begin{equation*}
\nabla^*\nabla u=\eps^{-2}h(u)u,\qquad \eps^2d^*F_\nabla=j_\nabla u.
\end{equation*}
By the assumption $|u|\leq1$, we have $0\leq h\leq\frac12$. Define the real two-form
\begin{equation*}
K(X,Y)\coloneqq-2\langle\nabla_Xu,i\nabla_Yu\rangle.
\end{equation*}
Direct differentiation and \cite[(2.7)]{Pigati-Stern2021} give
\begin{equation}\label{eq: current identities}
d(j_\nabla u)=K-|u|^2F_\nabla,\qquad d|u|^2\wedge j_\nabla u=|u|^2K,\qquad |K|\leq|\nabla u|^2.
\end{equation}
The scalar and curvature equations take the form \cite[(3.4), (2.6)]{Pigati-Stern2021}
\begin{equation}\label{eq: scalar curvature equations}
(-\eps^2\Delta+|u|^2)h=\eps^2|\nabla u|^2,\qquad (-\eps^2\Delta+|u|^2)F_\nabla=K.
\end{equation}

\subsection{Discrepancy, interior estimates and compactness}\label{sec: compactness}

Here we collect the interior estimates and compactness results needed below, mainly from \cite{Pigati-Stern2021}. Fix $\Omega\subset M^3$ and $\Omega'\Subset\Omega$, and consider smooth critical points with $|u|\leq1$ and $E_\eps(u,\nabla;\Omega)\leq\Lambda$. 

The key to the $(n-2)$-monotonicity of the energy, and consequent compactness, is the upper bound on the discrepancy 
\begin{equation}\label{eq: discrepancy definition}
\xi_\eps\coloneqq\eps|F_\nabla|-\eps^{-1}h(u).
\end{equation}
We write $\xi_\eps^\pm=\max\{\pm\xi_\eps,0\}$ for its positive and negative parts. The required upper bound was proved in \cite[Proposition 4.2]{Pigati-Stern2021}. We record the following interior versions of \cite[Proposition 4.2, Theorem 4.3 and Proposition 5.1]{Pigati-Stern2021}, together with the curvature derivative estimate.  For sufficiently small $\eps$, there are constants $C,r_0>0$, depending only on the domains, their geometry and $\Lambda$, such that on $\Omega'$
\begin{equation}\label{eq: PS discrepancy}
\xi_\eps\leq C,
\end{equation}
\begin{equation}\label{eq: pointwise}
|\nabla u|^2\leq C\eps^{-2}h^2+C\eps^2,\qquad \eps^2|F_\nabla|^2\leq C\eps^{-2}h^2+C\eps,
\end{equation}
\begin{equation}\label{eq: curvature derivative}
|\nabla F_\nabla|\leq C\eps^{-3},
\end{equation}
and
\begin{equation}\label{eq: local energy}
E_\eps(u,\nabla;B_r(x))\leq Cr\qquad (x\in\Omega',\ 0<r<r_0).
\end{equation}

\begin{proof}[Justification of the interior estimates]
All estimates are obtained on successive fixed domains compactly contained in $\Omega$. The curvature equation and the two-form Weitzenb\"ock identity give
\begin{equation*}
-\Delta F_\nabla=\nabla^*\nabla F_\nabla+\mathcal R_2F_\nabla,
\end{equation*}
where $\mathcal R_2$ is the curvature endomorphism of two-forms. The discrepancy inequality \cite[(3.7)]{Pigati-Stern2021} and the positive-part chain rule imply
\begin{equation*}
\Delta\xi_\eps^+\geq-C\eps|F_\nabla|.
\end{equation*}
The energy bounds both $\xi_\eps^+$ and $\eps|F_\nabla|$ in $L^2$. Since $2>3/2$, the interior elliptic estimate for subsolutions gives \eqref{eq: PS discrepancy}. Together with $h\leq1/2$, this yields $|F_\nabla|\leq C\eps^{-2}$ on smaller domains.

The inner-variation calculation in \cite[Section 4]{Pigati-Stern2021} is local, and its discrepancy term satisfies
\begin{equation*}
\eps^2|F_\nabla|^2-\eps^{-2}h^2
=\xi_\eps(\eps|F_\nabla|+\eps^{-1}h)\leq Ce_\eps^{1/2}.
\end{equation*}
The monotonicity argument of \cite[Theorem 4.3]{Pigati-Stern2021}, applied inside a fixed larger domain, therefore gives \eqref{eq: local energy}. The Bochner inequality \cite[(5.4)]{Pigati-Stern2021} implies $\Delta|\nabla u|^2\geq-C\eps^{-2}|\nabla u|^2$. The interior mean-value estimate on $B_\eps(x)$ and \eqref{eq: local energy} yield
\begin{equation*}
|\nabla u|^2(x)\leq C\eps^{-3}E_\eps(u,\nabla;B_\eps(x))\leq C\eps^{-2}.
\end{equation*}
On rescaled $\eps$-balls, $\eps^2F_\nabla$ and the source $\eps^2K$ are uniformly bounded. Interior $W^{2,q}$ estimates for the curvature equation, with $q>3$, then give \eqref{eq: curvature derivative}.

To localise the sharper bounds in \eqref{eq: pointwise}, set $v=|\nabla u|-2\eps^{-1}h$. The calculation in \cite[(5.5)--(5.6)]{Pigati-Stern2021} gives, on $\{v>0\}$ and for small $\eps$,
\begin{equation*}
\eps^2\Delta v\geq(1+3h-C\eps)v+2\eps v^2+2\eps^{-1}h^2-2Ch
\geq\tfrac12v-C\eps.
\end{equation*}
Thus $(v-C\eps)_+$ satisfies a screened subsolution inequality. Its coarse bound $C\eps^{-1}$ and comparison with exponential barriers on fixed normal balls give $v\leq C\eps$ on smaller domains, proving the gradient bound in \eqref{eq: pointwise}. Similarly, the calculation in \cite[proof of Proposition 5.1]{Pigati-Stern2021} shows that the positive part of
\begin{equation*}
\eps|F_\nabla|-(1+C\eps)\eps^{-1}h-C\sqrt\eps
\end{equation*}
satisfies a screened subsolution inequality at scale $\eps^{3/4}$. The same interior comparison, as in \cite[proof of Proposition 5.3]{Pigati-Stern2021}, gives
$\eps|F_\nabla|\leq(1+C\eps)\eps^{-1}h+C\sqrt\eps$ on smaller domains. Squaring proves the curvature bound in \eqref{eq: pointwise}.
\end{proof}

We also record the localised version of the compactness statement in dimension three.

\begin{theorem}[Pigati--Stern]\label{thm: PS21}
Let $\eps_j\to 0$ and let $(u_j,\nabla_j)$ be smooth critical points in $\Omega$ with $|u_j|\leq1$ and $E_{\eps_j}(u_j,\nabla_j;\Omega)\leq\Lambda$. After passing to a subsequence, there is a stationary integral one-varifold $V$ in $\Omega$ such that
\begin{equation*}
e_{\eps_j}(u_j,\nabla_j)\,d\vol_g\rightharpoonup\mu\coloneqq2\pi\|V\|
\end{equation*}
as locally finite measures. In particular, $\mu$ is singular with respect to volume. If $\tau$ is either unit orientation of its tangent line, defined $\mu$-almost everywhere, then
\begin{equation}\label{eq: stress convergence}
T_{\eps_j}\,d\vol_g\rightharpoonup\tau^\flat\otimes\tau^\flat\,\mu,\qquad T_\eps\coloneqq e_\eps g-2S-2\eps^2F^\top F,
\end{equation}
where $S(X,Y)\coloneqq\langle\nabla_Xu,\nabla_Yu\rangle$.
\end{theorem}
\begin{proof}
This is the interior form of \cite[Theorem 1.1 and Proposition 6.4]{Pigati-Stern2021}, with the normalisation \eqref{eq: energy}. Once the preceding interior estimates are available, the clearing-out, decay, compactness and integrality arguments in Sections 4--6 of that paper apply on successively smaller domains. In the decay argument, the distance to the vortex region is truncated by the distance to the boundary of a fixed larger interior domain. Exhaustion and a diagonal subsequence give the stated local convergence.
\end{proof}

\subsection{Scalar coercivity}
The following coercivity estimate follows from the local energy bound, which ensures that $|u|^2\geq1/2$ on a fixed positive proportion of every ball of radius $R\eps$, for a sufficiently large fixed $R$.

\begin{lemma}\label{lem: coercivity}
For every $\Omega'\Subset\Omega$, all sufficiently small $\eps$ and every real $z\in H^1_0(\Omega')$, extended by zero to $\Omega$,
\begin{equation*}
\int_\Omega z^2\,d\vol_g\leq C\int_\Omega\big(\eps^2|dz|^2+|u|^2z^2\big)\,d\vol_g.
\end{equation*}
\end{lemma}
\begin{proof}
On $\{|u|^2<\tfrac12\}$ the potential energy density is bounded below by $\frac1{16}\eps^{-2}$. For a fixed $R>1$, \eqref{eq: local energy} therefore gives
\begin{equation*}
\vol_g\big(B_{R\eps}(x)\cap\{|u|^2<1/2\}\big)\leq C\eps^2 E_\eps(u,\nabla;B_{R\eps}(x)) \leq CR\eps^3.
\end{equation*}
On the other hand we have the volume lower bound $\vol_g(B_{R\eps}(x))\geq cR^3\eps^3$, for some $c>0$. Thus, we can pick $R$ such that, denoting $G\coloneqq B_{R\eps}(x)\cap\{|u|^2\geq\tfrac12\}$ and $B\coloneqq B_{R\eps}(x)$, we have 
\begin{equation}\label{eq: volume G}
    \vol_g(G)\geq \frac12\vol_g(B)
\end{equation}
uniformly in $\eps$ and $x\in\Omega'$.

Let $z_B$ be the average of $z$ on $B$. The lower bound \eqref{eq: volume G} implies that 
\begin{equation*}
    \vol_g(B)|z_B|^2\leq 2\vol_g(G)|z_B|^2\leq 4\int_G z^2+4\int_G|z-z_B|^2
\end{equation*}
Thus, by Poincar\'e's inequality and the fact that $1\leq 2|u|^2$ on $G$,
\begin{equation*}
    \begin{split}
        \int_{B}z^2 d\vol_g&=\int_B|z-z_B|^2+ \vol_g(B)|z_B|^2\\
        &\leq 5\int_B|z-z_B|^2+4\int_G|z|^2\\
        &\leq CR^2\eps^2\int_B|dz|^2+C\int_B |u|^2|z|^2
    \end{split}
\end{equation*}
Since $R$ is fixed and uniform in $\eps$, it can be absorbed in the constant. A covering argument concludes the proof.
\end{proof}

\subsection{Stability and covariant test pairs}

Here we adapt the stability tests used in \cite{Badran2026b} to the manifold setting, which leaves a geometrically controlled error.

Recall that the second variation of \eqref{eq: energy} is
\begin{equation*}
\begin{split}
Q_{\eps,{\bf U}}[\b\Phi]=\int_\Omega\big(&|\nabla\phi-i\Phi u|^2-2\langle\nabla u,i\Phi\phi\rangle+\eps^2|d\Phi|^2\\
&+\eps^{-2}\langle u,\phi\rangle^2-\eps^{-2}h(u)|\phi|^2\big)\,d\vol_g.
\end{split}
\end{equation*}
Stability of ${\bf U}$ means $Q_{\eps,{\bf U}}[\b\Phi]\geq0$ for every smooth compactly supported pair $\b\Phi$. It is natural to consider the gauge-corrected form
\begin{equation*}
\begin{split}
\Q_{\eps,{\bf U}}[\b\Phi]\coloneqq\int_\Omega\big(&|\nabla\phi|^2+\eps^2|d\Phi|^2+\eps^2|d^*\Phi|^2+4\langle i\Phi\cdot\nabla u,\phi\rangle\\
&+\eps^{-2}(|u|^2-h(u))|\phi|^2+|u|^2|\Phi|^2\big)\,d\vol_g,
\end{split}
\end{equation*}
which satisfies the exact identity
\begin{equation}\label{eq: gauge correction}
\begin{split}
\Q_{\eps,{\bf U}}[\b\Phi]&=Q_{\eps,{\bf U}}[\b\Phi]+\eps^2\int_\Omega G_{\eps,{\bf U}}[\b\Phi]^2\,d\vol_g,\\
G_{\eps,{\bf U}}[\b\Phi]&\coloneqq d^*\Phi+\eps^{-2}\langle iu,\phi\rangle.
\end{split}
\end{equation}

We denote the bilinear form associated to $Q_{\eps,{\bf U}}$ by $\mathcal B_{\eps,{\bf U}}$ and the gauge-corrected linearised operator by $L_{\eps,{\bf U}}$, as in \cite{Badran2026b}. They are related by
\begin{equation*}
\begin{split}
\int_\Omega L_{\eps,{\bf U}}[\b\Phi]\cdot\b\Psi\,d\vol_g={}&\mathcal B_{\eps,{\bf U}}[\b\Phi,\b\Psi]+\eps^2\int_\Omega G_{\eps,{\bf U}}[\b\Phi]G_{\eps,{\bf U}}[\b\Psi]\,d\vol_g.
\end{split}
\end{equation*}
All the formulas above extend to compactly supported $H^1$ pairs by density. For the rest of the proof we omit $\eps,{\bf U}$ from $Q,\Q,\mathcal B,G$ and $L$.

On a neighbourhood of the compact supports under consideration, choose a smooth isometric embedding $\iota$ into $\R^N$ and set $X_a=\nabla\iota^a$. The projected coordinate fields satisfy
\begin{equation}\label{eq: projected fields}
\sum_aX_a\otimes X_a=g^{-1},\qquad \sum_aX_a\otimes\nabla_YX_a=0
\end{equation}
for every vector field $Y$. The first identity expresses isometry. For the second, let $E_a$ be the standard basis of $\R^N$ and $\mathrm{II}$ the second fundamental form of $\iota$. For vector fields $Z,W$,
\begin{equation*}
\langle X_a,Z\rangle=\langle d\iota(Z),E_a\rangle,\qquad \langle\nabla_YX_a,W\rangle=\nabla^2\iota^a(Y,W)=\langle\mathrm{II}(Y,W),E_a\rangle,
\end{equation*}
and hence
\begin{equation*}
\sum_a\langle X_a,Z\rangle\langle\nabla_YX_a,W\rangle=\langle d\iota(Z),\mathrm{II}(Y,W)\rangle=0.
\end{equation*}
Define the fixed symmetric tensor $H\in\Gamma(\operatorname{Sym}^2T^*M)$ by
\begin{equation*}
H_{ij}\coloneqq\sum_{a,k}\langle\nabla_{e_k}X_a,e_i\rangle\langle\nabla_{e_k}X_a,e_j\rangle.
\end{equation*}

\begin{lemma}\label{lem: tensor identity}
For every compactly supported $T\in H^1(\Omega;T^*M\otimes T^*M)$,
\begin{equation}\label{eq: tensor identity}
\begin{split}
\sum_a\int_\Omega\big(|d(T(X_a,\cdot))|^2&+|d^*(T(X_a,\cdot))|^2\big)\,d\vol_g=\int_\Omega|\nabla T|^2\,d\vol_g\\
&+\int_\Omega\sum_i H\big(T(\cdot,e_i)^\sharp,T(\cdot,e_i)^\sharp\big)\,d\vol_g\\
&+\int_\Omega\sum_i\Ric\big(T(e_i,\cdot)^\sharp,T(e_i,\cdot)^\sharp\big)\,d\vol_g,
\end{split}
\end{equation}
where $\{e_i\}$ is any local orthonormal frame.
\end{lemma}
\begin{proof}
For smooth compactly supported $T$, set $t_a=T(X_a,\cdot)$. The one-form Weitzenb\"ock identity \cite[Theorem 9.4.1]{Petersen2016}, integrated over $\Omega$, gives
\begin{equation*}
\int_\Omega\big(|dt_a|^2+|d^*t_a|^2\big)\,d\vol_g=\int_\Omega\big(|\nabla t_a|^2+\Ric(t_a^\sharp,t_a^\sharp)\big)\,d\vol_g.
\end{equation*}
Expand $\nabla t_a=(\nabla T)(X_a,\cdot)+T(\nabla X_a,\cdot)$ and sum in $a$. The mixed term vanishes by \eqref{eq: projected fields}, so
\begin{equation*}
\begin{split}
\sum_a|\nabla t_a|^2&=|\nabla T|^2+\sum_i H\big(T(\cdot,e_i)^\sharp,T(\cdot,e_i)^\sharp\big),\\
\sum_a\Ric(t_a^\sharp,t_a^\sharp)&=\sum_i\Ric\big(T(e_i,\cdot)^\sharp,T(e_i,\cdot)^\sharp\big).
\end{split}
\end{equation*}
Approximation gives the identity for $H^1$ tensors.
\end{proof}
Now, set
\begin{equation*}
b\coloneqq(*F_\nabla)^\sharp,\qquad \nu\coloneqq\frac{b}{|b|}\quad\text{on }\{|F_\nabla|>0\},\qquad B\coloneqq\sqrt{-F^2}.
\end{equation*}
In dimension three,
\begin{equation}\label{eq: B identities}
B=|F_\nabla|(g-\nu^\flat\otimes\nu^\flat),\qquad FF^\top=F^\top F=B^2,\qquad \|\nabla F\|^2=\|\nabla B\|^2.
\end{equation}
We identify $B$ with the symmetric bilinear form obtained by lowering an index with $g$, and write $B^2(X,Y)=\langle B(X,\cdot),B(Y,\cdot)\rangle$. We extend $B$ by zero on $\{|F_\nabla|=0\}$. The map $b\mapsto |b|I-b\otimes \frac{b}{|b|}$, with this extension, is Lipschitz, so $B\in H^1_{\mathrm{loc}}(\Omega;\operatorname{Sym}^2T^*M)$.

We write $\dagger$ for the conjugate transpose in an orthonormal frame. As in \cite{Badran2026b}, consider the pairs
\begin{equation*}
{\bf V}_a\coloneqq(\nabla_{X_a}u,\iota_{X_a}F_\nabla),\qquad {\bf W}_a\coloneqq(i\nabla_{X_a}u,-B(X_a,\cdot)).
\end{equation*}

\begin{proposition}\label{prop: comparison}
For every real $\zeta\in C_c^\infty(\Omega)$,
\begin{equation}\label{eq: comparison}
\sum_a\Q[\zeta{\bf W}_a]+4\int_\Omega\zeta^2(\nabla u)^\dagger(B+iF)(\nabla u)\,d\vol_g=\sum_a\Q[\zeta{\bf V}_a].
\end{equation}
Moreover, the Hermitian endomorphism $B+iF$ is positive semidefinite, thus if the critical point is stable in $\Omega$, then
\begin{equation}\label{eq: test bound}
0\leq\sum_a\Q[\zeta{\bf W}_a]\leq C\int_\Omega(\zeta^2+|d\zeta|^2)e_\eps(u,\nabla)\,d\vol_g.
\end{equation}
\end{proposition}
\begin{proof}
Apply the calculations of \cite[Section 3.1, equations (3.5)--(3.8)]{Badran2026b} and \Cref{lem: tensor identity} to the compactly supported covariant tensors $\zeta F_\nabla$ and $\zeta B$. By \eqref{eq: B identities}, they have the same norms, derivative norms and geometric contributions. The difference of the mixed terms is $-4\zeta^2(\nabla u)^\dagger(B+iF)(\nabla u)$, proving \eqref{eq: comparison}.

To estimate the right-hand side, use the Bochner calculations for $\nabla u$ and $F_\nabla$ in \cite[Section 3, equation (3.1) and the calculation following (3.4)]{Pigati-Stern2021}. The terms involving $\langle F_\nabla,K\rangle$ cancel. The scalar Laplacian term is
\begin{equation*}
\frac12\int_\Omega\zeta^2\Delta\big(|\nabla u|^2+2\eps^2|F_\nabla|^2\big)\,d\vol_g,
\end{equation*}
which cancels the mixed cutoff term after integration by parts. Consequently,
\begin{equation*}
\begin{split}
\sum_a\Q[\zeta{\bf V}_a]={}&\int_\Omega\zeta^2\big(\langle H-\Ric,S\rangle+\eps^2\langle H+\Ric,B^2\rangle\\
&\hspace{5em}-2\eps^2\langle\mathcal R_2F_\nabla,F_\nabla\rangle\big)\,d\vol_g\\
&+\int_\Omega|d\zeta|^2\big(|\nabla u|^2+2\eps^2|F_\nabla|^2\big)\,d\vol_g.
\end{split}
\end{equation*}
Finally, $B+iF=|iF|+iF$ is positive semidefinite. Stability for the compactly supported pairs $\zeta{\bf W}_a$ and boundedness of the fixed geometric tensors give \eqref{eq: test bound}.
\end{proof}

\section{The discrepancy bound}\label{sec: discrepancy}

Recall Pigati--Stern's discrepancy $\xi_\eps$ from \eqref{eq: discrepancy definition}. In the notation of \cite[Section 3.1]{Badran2026b},
\begin{equation*}
-\eps\xi_\eps=h(u)-\frac{\eps^2}{2}\operatorname{tr}B=h(u)-\eps^2|F_\nabla|.
\end{equation*}
The pointwise estimate \eqref{eq: PS discrepancy} from \cite{Pigati-Stern2021} controls the positive part $\xi_\eps^+$. Here we show that in dimension three \emph{stability} controls the negative part $\xi_\eps^-$, including its derivative, in an integral way. 

\begin{theorem}[Interior discrepancy estimate]\label{thm: discrepancy}
For every smooth critical point in $\Omega$ with $|u|\leq1$ and $E_\eps(u,\nabla;\Omega)\leq\Lambda$, every $\Omega'\Subset\Omega$, and all sufficiently small $\eps$,
\begin{equation}\label{eq: positive discrepancy}
\|\eps^{-1}\xi_\eps^+\|_{L^2(\Omega')}+\|d\xi_\eps^+\|_{L^2(\Omega')}\leq C\eps.
\end{equation}
If the critical point is stable with respect to compactly supported variations in $\Omega$, then
\begin{equation}\label{eq: negative discrepancy}
\|\eps^{-1}\xi_\eps\|_{L^2(\Omega')}+\|d\xi_\eps\|_{L^2(\Omega')}\leq C.
\end{equation}
Moreover, for a sequence of stable critical points whose energy measures converge locally to $\mu$ as in \Cref{thm: PS21},
\begin{equation}\label{eq: strong scalar}
\|\eps^{-1}\xi_\eps\|_{L^1(\Omega')}\longrightarrow0,
\end{equation}
and, as locally finite measures in $\Omega$,
\begin{equation}\label{eq: scalar mass}
\eps^{-2}h\,d\vol_g\rightharpoonup\mu,\qquad |F_\nabla|\,d\vol_g\rightharpoonup\mu.
\end{equation}
\end{theorem}

\begin{proof}
    First, we show that every critical point satisfies \eqref{eq: positive discrepancy}, which is just a consequence of the discrepancy bound by \cite{Pigati-Stern2021}. Indeed, by \cite[Equation (3.7)]{Pigati-Stern2021}, we have 
    \begin{equation*}
(-\eps^2\Delta+|u|^2)\xi_\eps\leq C\eps^3|F_\nabla|. 
\end{equation*}
For $\eta\in C_c^\infty(\Omega)$, testing against $\eta^2\xi_\eps^+$ gives
\begin{equation*}
\begin{split}
\int_\Omega\big(\eps^2|d(\eta\xi_\eps^+)|^2+|u|^2\eta^2(\xi_\eps^+)^2\big)\,d\vol_g
\leq{}&C\eps^3\int_\Omega\eta^2|F_\nabla|\xi_\eps^+\,d\vol_g\\
&+\eps^2\int_\Omega|d\eta|^2(\xi_\eps^+)^2\,d\vol_g.
\end{split}
\end{equation*}
\Cref{lem: coercivity}, the energy bound and Young's inequality yield
\begin{equation*}
\|\eta\xi_\eps^+\|_{L^2}^2+\eps^2\|d(\eta\xi_\eps^+)\|_{L^2}^2\leq C\eps^4+C\eps^2\int_\Omega|d\eta|^2(\xi_\eps^+)^2\,d\vol_g.
\end{equation*}
Initially, the energy bound gives $\|\xi_\eps^+\|_{L^2(\Omega)}^2\leq C$. Apply the preceding inequality with two nested cutoffs, the inner one equal to one on $\Omega'$ and supported where the outer one is one. The first application gives an interior squared $L^2$ bound $C\eps^2$ for $\xi_\eps^+$, and the second gives $C\eps^4$, together with $\|d\xi_\eps^+\|_{L^2(\Omega')}^2\leq C\eps^2$. This proves \eqref{eq: positive discrepancy}.

\medskip

Next, we assume that the critical points $(u,\nabla)$ are stable. The proof of \eqref{eq: negative discrepancy} can be understood as a quantitative version of the estimates in \cite{Badran2026b}. 

For every real $\psi\in H^1(\Omega)$ with compact support, set $\b\Psi_a=(0,\psi X_a^\flat)$. Then, using the polarisation of \eqref{eq: tensor identity}, we find
\begin{equation*}
    \begin{split}
        \sum_a\int_\Omega L[{\bf W}_a]\cdot\b\Psi_a\,d\vol_g
        ={}&-\eps^2\int_\Omega\langle d(\tr B),d\psi\rangle\,d\vol_g\\
        &+\int_\Omega\psi(2|\nabla u|^2-|u|^2\tr B)\,d\vol_g\\
        &-\eps^2\int_\Omega\psi\langle H+\Ric,B\rangle\,d\vol_g.
    \end{split}
\end{equation*}
Testing the scalar equation in \eqref{eq: scalar curvature equations} against $2\psi$ and using that $2h-\eps^2\tr B=-2\eps \xi_\eps$, we get
\begin{equation}\label{eq: trace}
\begin{split}
\sum_a\int_\Omega L[{\bf W}_a]\cdot\b\Psi_a\,d\vol_g={}&-2\eps\int_\Omega\big(\langle d\xi_\eps,d\psi\rangle+\eps^{-2}|u|^2\xi_\eps\psi\big)\,d\vol_g\\
&-\eps^2\int_\Omega\psi\langle H+\Ric,B\rangle\,d\vol_g.
\end{split}
\end{equation}
Moreover, by applying \Cref{lem: tensor identity} to $T=\psi g$, we get
\begin{equation}\label{eq: scalar test form}
\sum_a\Q[\b\Psi_a]=\int_\Omega\big(3\eps^2|d\psi|^2+3|u|^2\psi^2+\eps^2(\operatorname{tr}H+\Scal)\psi^2\big)\,d\vol_g.
\end{equation}

\medskip

Choose $0\leq\eta,\zeta\leq1$ in $C_c^\infty(\Omega)$, with $\eta=1$ on $\Omega'$ and $\zeta=1$ on a neighbourhood of $\operatorname{supp}\eta$. Set $\psi=-\eps^{-1}\eta^2\xi_\eps$. The pairing of $\zeta{\bf W}_a$ with $\b\Psi_a$ equals that of ${\bf W}_a$ with $\b\Psi_a$, since $d\zeta=0$ near its support. By \eqref{eq: test bound}, Cauchy--Schwarz for $\Q$, \eqref{eq: scalar test form} and \Cref{lem: coercivity},
\begin{equation*}
\begin{split}
\left|\sum_a\int_\Omega L[{\bf W}_a]\cdot\b\Psi_a\,d\vol_g\right|
&\leq\left(\sum_a\Q[\zeta{\bf W}_a]\right)^{1/2}\left(\sum_a\Q[\b\Psi_a]\right)^{1/2}\\
&\leq C\big(\|d(\eta\xi_\eps)\|_{L^2}+\|\eps^{-1}|u|\eta\xi_\eps\|_{L^2}\big).
\end{split}
\end{equation*}
The geometric term in \eqref{eq: trace} is bounded by
\begin{equation*}
\begin{split}
C\eps\int_\Omega\eta^2|F_\nabla||\xi_\eps|\,d\vol_g
&\leq C\eps\|\eps^{-1}\eta\xi_\eps\|_{L^2}\\
&\leq C\eps\big(\|d(\eta\xi_\eps)\|_{L^2}+\|\eps^{-1}|u|\eta\xi_\eps\|_{L^2}\big).
\end{split}
\end{equation*}
The energy bound gives $\int_\Omega|d\eta|^2|\xi_\eps|^2\,d\vol_g\leq C$. Using
\begin{equation*}
\langle d\xi_\eps,d(\eta^2\xi_\eps)\rangle=|d(\eta\xi_\eps)|^2-|d\eta|^2(\xi_\eps)^2,
\end{equation*}
equation \eqref{eq: trace} therefore yields
\begin{equation*}
\begin{split}
\|d(\eta\xi_\eps)\|_{L^2}^2+\|\eps^{-1}|u|\eta\xi_\eps\|_{L^2}^2
\leq C\big(\|d(\eta\xi_\eps)\|_{L^2}+\|\eps^{-1}|u|\eta\xi_\eps\|_{L^2}\big)+C.
\end{split}
\end{equation*}
Both norms are bounded. \Cref{lem: coercivity} supplies the unweighted $L^2$ bound for $\eps^{-1}\eta\xi_\eps$, and $\eta=1$ on $\Omega'$ proves \eqref{eq: negative discrepancy}.

\medskip

Lastly, we prove \eqref{eq: strong scalar} and \eqref{eq: scalar mass}. 
Equation \eqref{eq: scalar curvature equations} gives, for every $0\leq\chi\in C_c^\infty(\Omega)$,
\begin{equation*}
\int_\Omega\chi\eps^{-2}h\,d\vol_g=\int_\Omega\chi\big(|\nabla u|^2+2\eps^{-2}h^2\big)\,d\vol_g+\int_\Omega h\Delta\chi\,d\vol_g\leq C_\chi.
\end{equation*}
Since $h\geq0$, the measures $\eps^{-2}h\,d\vol_g$ have locally bounded mass. In particular, $h\to0$ in $L^1_{\mathrm{loc}}$, so $\Delta h\to0$ distributionally. The same equation also gives
\begin{equation*}
e_\eps-\eps^{-2}h=-\Delta h+\eps^{-1}\xi_\eps(h+\eps^2|F_\nabla|).
\end{equation*}
By \eqref{eq: negative discrepancy}, $\eps^{-1}\xi_\eps$ is bounded in $L^2_{\mathrm{loc}}$, while $\|h+\eps^2|F_\nabla|\|_{L^2(\Omega)}\leq C\eps$ by the energy bound, hence the last term tends to zero in $L^1_{\mathrm{loc}}$. Hence $\eps^{-2}h\,d\vol_g\rightharpoonup\mu$ locally.

The negative part $\eps^{-1}\xi_\eps^-$ is locally bounded in $L^2$ and satisfies $0\leq\eps^{-1}\xi_\eps^-\leq\eps^{-2}h$. Every weak $L^2$ limit on an interior open set therefore defines an absolutely continuous nonnegative measure dominated by $\mu$. Since $\mu$ is singular with respect to volume, the $L^2$ limit is zero. Testing weak $L^2$ convergence against the indicator of $\Omega'$ gives $\|\eps^{-1}\xi_\eps^-\|_{L^1(\Omega')}\to0$. The rescaled positive part $\eps^{-1}\xi_\eps^+$ tends to zero in $L^2(\Omega')$ by \eqref{eq: positive discrepancy}, proving \eqref{eq: strong scalar}. Finally, $|F_\nabla|=\eps^{-2}h+\eps^{-1}\xi_\eps$ gives \eqref{eq: scalar mass}.
\end{proof}

\section{Geometric consequences}
Fix a sequence of smooth critical points with $\eps\to0$, $|u|\leq1$ and uniformly bounded energy, stable under compactly supported variations in $\Omega$. Pass to a subsequence defining $\mu=2\pi\|V\|$ as in \Cref{thm: PS21}. For the estimates below, we introduce a modified Jacobian, as in \cite[(III.23)]{Bethuel-Brezis-Orlandi2001} and \cite[Section 2]{Pigati-Stern2023}.

Fix $b_0\in(0,1)$ and a smooth nondecreasing function $t\colon[0,\infty)\to[0,1]$ which is zero near zero and equals one on $[b_0,\infty)$. Define $\chi(a)=t(a)/a$, extended smoothly by zero near $a=0$, and set
\begin{equation*}
J_\eps\coloneqq F_\nabla+d\big(\chi(|u|^2)j_\nabla u\big),\qquad
q_\eps\coloneqq(*J_\eps)^\sharp.
\end{equation*}
By \eqref{eq: current identities},
\begin{equation}\label{eq: core vorticity formula}
J_\eps=(1-t(|u|^2))F_\nabla+t'(|u|^2)K.
\end{equation}
Consequently,
\begin{equation*}
dJ_\eps=0,\qquad \div q_\eps=0,\qquad \spt q_\eps\subset\{|u|^2\leq b_0\}.
\end{equation*}
On this core, $h\geq c_0\coloneqq(1-b_0)/2>0$. We first state the integral curvature estimate for $q_\eps$.

\begin{proposition}[Integral curvature estimates]\label{prop: integral curvature estimates}
Fix $\Omega'\Subset\Omega$. For all sufficiently small $\eps$, $\nu$ is smooth near $\spt q_\eps\cap\overline{\Omega'}$. For every normal ball $B_{2r}(x)\subset\Omega'$ with $0<2r<r_0$,
\begin{equation}\label{eq: turning estimate}
\int_{B_r(x)}|\nabla_{q_\eps}\nu|\,d\vol_g\leq CE_\eps(u,\nabla;B_{2r}(x))^{1/2}(1+\eps/r)+C\eps E_\eps(u,\nabla;B_{2r}(x)).
\end{equation}
The integrand is extended by zero where $q_\eps=0$. In particular,
\begin{equation}\label{eq: turning limit}
\limsup_{\eps\to0}\int_{B_r(x)}|\nabla_{q_\eps}\nu|\,d\vol_g\leq C\sqrt r,
\end{equation}
with constant depending on $\Omega',\Lambda$.
\end{proposition}

The preceding proposition controls the turning of $\nu=(*F_\nabla)^\sharp/|F_\nabla|$ along $q_\eps$. To relate this estimate to the limit varifold, we show that $\nu$ recovers its tangent directions and that $q_\eps$ carries its full unsigned mass. This follows from \Cref{thm: discrepancy} and the first-order identities used in \cite{Badran2026b}.

\begin{proposition}\label{prop: stress jacobian approximation}
For any measurable unit extension of $\nu$ to $\{F_\nabla=0\}$ and every $\Omega'\Subset\Omega$,
\begin{equation*}
\|T_\eps-e_\eps\nu^\flat\otimes\nu^\flat\|_{L^1(\Omega')}\longrightarrow0,
\end{equation*}
where $T_\eps$ is the stress-energy tensor. Decompose
\begin{equation}\label{eq: q decomposition}
q_\eps=\alpha_\eps\nu+R_\eps,\qquad R_\eps\perp\nu.
\end{equation}
Then $|q_\eps|\leq Ce_\eps$, and, locally as Radon measures in $\Omega$,
\begin{equation*}
(\alpha_\eps)_+\,d\vol_g\rightharpoonup\mu,\qquad |q_\eps|\,d\vol_g\rightharpoonup\mu.
\end{equation*}
Moreover, for every $\Omega'\Subset\Omega$,
\begin{equation*}
\int_{\Omega'}\big((\alpha_\eps)_-+|R_\eps|\big)\,d\vol_g\longrightarrow0.
\end{equation*}
\end{proposition}

The preceding two propositions imply the main result.

\begin{theorem}\label{thm: main}
    For the sequence and hypotheses fixed at the start of this section, the limit varifold $V$ is locally a finite union of smooth geodesic arcs without interior endpoints, counted with positive integer multiplicities and possibly intersecting. If $\Omega=M$ is closed, $V$ is a finite union of closed immersed geodesics, counted with positive integer multiplicities.
\end{theorem}
\begin{proof}
Pigati--Stern compactness and \cite[Section 3]{Allard-Almgren1976} give a locally finite geodesic network in $\Omega$. We show that its outgoing directions pair antipodally at every interior vertex. Consider $\tau$ from \eqref{eq: stress convergence}. Now let $S\Omega$ be the unit tangent bundle over $\Omega$, and let $\Xi(x,v)\in T_x^*M$ be continuous on $S\Omega$, compactly supported in the base variable, and odd in $v$. Then we claim:
\begin{equation*}
\lim_{\eps\to0}\int_\Omega\Xi(x,\nu)(q_\eps)\,d\vol_g=\int_\Omega\Xi(x,\tau)(\tau)\,d\mu.
\end{equation*}
On the bundle $\mathbb P(T\Omega)$ of unoriented unit directions, consider
\begin{equation*}
e_\eps(x)\,d\vol_g(x)\,\delta_{[\nu(x)]}.
\end{equation*}
These measures have uniformly bounded mass over compact subsets of the base. A subsequential local weak limit has base measure $\mu$. Disintegrate it as $\mu(dx)\sigma_x(d[v])$, where $\sigma_x$ is a probability measure for $\mu$-almost every $x$. By \eqref{eq: stress convergence} and \Cref{prop: stress jacobian approximation}, for $\mu$-almost every $x$ we have
\begin{equation*}
\int v^\flat\otimes v^\flat\,d\sigma_x([v])=\tau^\flat\otimes\tau^\flat,
\qquad
\int\big(1-\langle v,\tau\rangle^2\big)\,d\sigma_x([v])=0.
\end{equation*}
Thus $\sigma_x=\delta_{[\tau(x)]}$, identifying every subsequential limit.

Since $0\leq(\alpha_\eps)_+\leq Ce_\eps$, any local weak limit of $(\alpha_\eps(x))_+\,d\vol_g(x)\,\delta_{[\nu(x)]}$ is dominated by $C$ times the preceding limit. Its base measure is $\mu$ by \Cref{prop: stress jacobian approximation}, so it equals $\mu(dx)\delta_{[\tau(x)]}$. Finally, oddness makes $\Xi(x,v)(v)$ a continuous function on $\mathbb P(T\Omega)$, and
\begin{equation*}
\int_\Omega\Xi(x,\nu)(q_\eps)\,d\vol_g-\int_\Omega(\alpha_\eps)_+\Xi(x,\nu)(\nu)\,d\vol_g\longrightarrow0
\end{equation*}
by the vanishing negative and transverse parts on the compact base support of $\Xi$. Apply the convergence just proved to the second integral. 

For $\varphi\in C_c^\infty(S\Omega)$, let $d_x^H\varphi$ and $d_v\varphi$ denote its horizontal and vertical derivatives. Where $\nu$ is smooth, the chain rule gives
\begin{equation*}
d(\varphi(x,\nu(x)))(q_\eps)=d_x^H\varphi(x,\nu)(q_\eps)+d_v\varphi(x,\nu)(\nabla_{q_\eps}\nu).
\end{equation*}
For sufficiently small $\eps$, $\nu$ is smooth near the intersection of $\spt q_\eps$ with the compact base support of $\varphi$. Multiply $\varphi(x,\nu(x))$ by a smooth cutoff supported where $\nu$ is smooth and equal to one near this intersection, then extend by zero. The derivative of the cutoff has zero pairing with $q_\eps$ wherever $\varphi$ can be nonzero. Since $\div q_\eps=0$,
\begin{equation}\label{eq: directional conservation}
\int_\Omega d_x^H\varphi(x,\nu)(q_\eps)\,d\vol_g=-\int_\Omega d_v\varphi(x,\nu)(\nabla_{q_\eps}\nu)\,d\vol_g.
\end{equation}

We claim that, at every vertex $x_0\in\Omega$ of the limiting network, the outgoing directions $v_i$ and positive integer multiplicities $m_i$ satisfy
\begin{equation}\label{eq: antipodal}
\sum_i m_i\delta_{v_i}=\sum_i m_i\delta_{-v_i}.
\end{equation}
Choose $r>0$ small enough that $\overline{B_r(x_0)}\subset\Omega$ and the network in this normal ball consists exactly of the finitely many incident geodesic arms. Given a smooth odd function $a\colon S_{x_0}M\to\R$, choose a smooth cutoff $\zeta$ equal to one near zero and zero near $[1,\infty)$, and set
\begin{equation*}
\varphi_r(x,v)\coloneqq\zeta(\dist(x,x_0)/r)a(P_{x\to x_0}v),
\end{equation*}
where $P_{x\to x_0}$ is radial parallel transport. Extend by zero outside the ball. Then $\varphi_r\in C_c^\infty(S\Omega)$ is odd in $v$, and $\|d_v\varphi_r\|_\infty\leq C_a$ independently of $r$.

For fixed $r$, apply the preceding calculations to $\Xi=d_x^H\varphi_r$. By \eqref{eq: directional conservation} and \eqref{eq: turning limit},
\begin{equation}\label{eq: odd test bound}
\left|\int_\Omega d_x^H\varphi_r(x,\tau)(\tau)\,d\mu\right|\leq C_a\sqrt r.
\end{equation}
The integrand is independent of the sign of $\tau$. On the $i$th outward arm $\gamma_i(t)=\exp_{x_0}(tv_i)$, radial parallel transport makes $\dot\gamma_i$ constant, so
\begin{equation*}
d_x^H\varphi_r(\gamma_i(t),\dot\gamma_i(t))(\dot\gamma_i(t))=r^{-1}\zeta'(t/r)a(v_i).
\end{equation*}
Its integral over the arm is $-a(v_i)$. Since $\mu=2\pi\|V\|$, \eqref{eq: odd test bound} gives
\begin{equation*}
2\pi\left|\sum_i m_i a(v_i)\right|\leq C_a\sqrt r.
\end{equation*}
Let $r\downarrow0$. For any smooth $b$, choose $a(v)=b(v)-b(-v)$ to obtain \eqref{eq: antipodal}. Pairing opposite arms with matching multiplicities and uniqueness for the geodesic equation give smooth continuation through every vertex, with no interior endpoints. When $\Omega=M$ is closed, local finiteness and compactness give a finite network. Resolving its vertices by these pairings decomposes the network, with multiplicity, into finitely many closed immersed geodesics. This proves the theorem.
\end{proof}

\subsection{The integral curvature estimates}

We now prove \Cref{prop: integral curvature estimates,prop: stress jacobian approximation} for the sequence fixed above.

Recall that $\nu=(*F_\nabla)^\sharp/|F_\nabla|$ on $\{|F_\nabla|>0\}$. One of the main tools in the analysis is the following defect:
\begin{equation*}
\mathcal S_\eps\coloneqq|\nabla u|^2-(*K)(\nu)+\eps^2|F_\nabla||\nabla\nu|^2
\end{equation*}
on this set, and extend it by $|\nabla u|^2$ where $F_\nabla=0$. In an oriented orthonormal frame with $e_3=\nu$, completing the square gives
\begin{equation}\label{eq: planar defect}
|\nabla u|^2-(*K)(\nu)=|\nabla_{e_3}u|^2+|\nabla_{e_1}u+i\nabla_{e_2}u|^2\geq0.
\end{equation}
In particular, $\mathcal S_\eps\geq0$. It vanishes in a neighbourhood if the Bogomolny equations hold in the curvature direction $\nu$ and $\nu$ is parallel. For critical points, the defect appears in the equation for the discrepancy: on $\{|F_\nabla|>0\}$,

\begin{equation}\label{eq: defect equation}
(\Delta-\eps^{-2}|u|^2)\xi_\eps=\eps^{-1}\mathcal S_\eps+\eps|F_\nabla|\Ric(\nu,\nu).
\end{equation}
Using the bounds already obtained for the discrepancy for stable critical pairs and \eqref{eq: defect equation}, we prove the following proposition which will be the key estimate in proving the localised integral curvature estimates of \Cref{prop: integral curvature estimates}.

\begin{proposition}[Defect estimates]\label{prop: curvature}

Let $(u,\nabla)$ be a smooth critical point in $\Omega$, stable under compactly supported variations, with $|u|\leq1$ and $E_\eps(u,\nabla;\Omega)\leq\Lambda$. For every $\Omega'\Subset\Omega$ and all sufficiently small $\eps$,
\begin{equation}\label{eq: localized defect}
\eps^{-1}\int_\Omega\eta h\mathcal S_\eps\,d\vol_g\leq CE_\eps(u,\nabla;\Omega_0)^{1/2}(1+\eps\|d\eta\|_{L^\infty})+C\eps E_\eps(u,\nabla;\Omega_0),
\end{equation}
for $0\leq\eta\leq1$ smooth with compact support in an open set $\Omega_0\subset\Omega'$. The constant depends on the fixed interior domains, their geometry and $\Lambda$, but not on $\Omega_0$, $\eta$ or $\eps$. Since $h\geq c_0$ on $\{|u|^2\leq b_0\}$, the same right-hand side, with a different constant, bounds
\begin{equation*}
\eps^{-1}\int_{\Omega\cap\{|u|^2\leq b_0\}}\eta\mathcal S_\eps\,d\vol_g.
\end{equation*}
Moreover,

\begin{align}\label{eq: auxiliary integral curvature}
    \int_{\Omega'\cap\{|F_\nabla|>0\}}\eps^2|F_\nabla|^2|\nabla_\nu\nu|^2\,d\vol_g \leq C.
\end{align}
Along the sequence defining $\mu$, we also have
\begin{equation}\label{eq: full defect convergence}
\int_{\Omega'}\mathcal S_\eps\,d\vol_g\longrightarrow0,
\end{equation}
and, for any measurable unit extension of $\nu$ to $\{F_\nabla=0\}$,
\begin{equation}\label{eq: first order convergence}
\int_{\Omega'}\big(|\nabla u|^2-(*K)(\nu)\big)\,d\vol_g\longrightarrow0.
\end{equation}
\end{proposition}
Before proving this proposition, we show how its estimates imply \Cref{prop: stress jacobian approximation,prop: integral curvature estimates}. First we identify the limiting mass of the modified Jacobian and its associated vector field $q_\eps$, and then the limit of the stress-energy tensor.
\begin{proof}[Proof of \Cref{prop: stress jacobian approximation}]
First note that:
\begin{equation}\label{eq: q coefficients}
\begin{split}
\alpha_\eps&=(1-t(|u|^2))|F_\nabla|+t'(|u|^2)(*K)(\nu),\\
R_\eps&=t'(|u|^2)\big((*K)^\sharp-(*K)(\nu)\nu\big).
\end{split}
\end{equation}
Where the coefficients in \eqref{eq: core vorticity formula} are nonzero, $h\geq c_0$ and $2h|F_\nabla|\leq e_\eps$. Together with $|K|\leq|\nabla u|^2$, this proves that $|q_\eps|\leq C e_\eps$. Since $1-t\geq0$ and $t'\geq0$,
\begin{equation*}
(\alpha_\eps)_-\leq C\big(|\nabla u|^2-(*K)(\nu)\big)\leq C\mathcal S_\eps.
\end{equation*}
The integral of the right-hand side over every interior domain tends to zero by \eqref{eq: full defect convergence}. In an oriented orthonormal frame with $e_3=\nu$, the components of $(*K)^\sharp$ perpendicular to $\nu$ contain $\nabla_{e_3}u$. Hence
\begin{equation}\label{eq: transverse bound}
|R_\eps|\leq C|\nabla u||\nabla_\nu u|\leq C|\nabla u|\big(|\nabla u|^2-(*K)(\nu)\big)^{1/2}.
\end{equation}
Integrating on $\Omega'$ and applying \eqref{eq: first order convergence}, the local energy bound and Cauchy--Schwarz, we obtain
\begin{equation*}
\int_{\Omega'}\big((\alpha_\eps)_-+|R_\eps|\big)\,d\vol_g\longrightarrow0.
\end{equation*}

Since $t(a)=a\chi(a)$, the scalar equation and the chain rule give
\begin{equation*}
\begin{split}
\frac12\Delta\left(\int_1^{|u|^2}\chi(a)\,da\right)={}&t'(|u|^2)|\nabla u|^2-\eps^{-2}t(|u|^2)h\\
&+\frac12\chi'(|u|^2)\big(|d|u|^2|^2-2|u|^2|\nabla u|^2\big).
\end{split}
\end{equation*}
Substitution in \eqref{eq: q coefficients} gives
\begin{equation}\label{eq: unsigned identity}
\begin{split}
\alpha_\eps={}&\eps^{-2}h+\frac12\Delta\left(\int_1^{|u|^2}\chi(a)\,da\right)+(1-t(|u|^2))\eps^{-1}\xi_\eps\\
&-t'(|u|^2)\big(|\nabla u|^2-(*K)(\nu)\big)\\
&-\frac12\chi'(|u|^2)\big(|d|u|^2|^2-2|u|^2|\nabla u|^2\big).
\end{split}
\end{equation}
The last term is controlled by the first-order defect:
\begin{equation*}
\big||d|u|^2|^2-2|u|^2|\nabla u|^2\big|\leq C|u|^2|\nabla u|\big(|\nabla u|^2-(*K)(\nu)\big)^{1/2}.
\end{equation*}
Indeed, in a local unitary trivialisation and an oriented orthonormal frame with $e_3=\nu$,
\begin{equation*}
\begin{split}
|d|u|^2|^2-2|u|^2|\nabla^Au|^2&=2\left\langle\sum_{i=1}^3(\nabla^A_{e_i}u)^2,u^2\right\rangle,\\
\sum_{i=1}^3(\nabla^A_{e_i}u)^2&=(\nabla^A_{e_1}u+i\nabla^A_{e_2}u)(\nabla^A_{e_1}u-i\nabla^A_{e_2}u)+(\nabla^A_{e_3}u)^2.
\end{split}
\end{equation*}
The first-order defect is given by \eqref{eq: planar defect}.
Since $a\chi'(a)$ is bounded, \eqref{eq: strong scalar}, \eqref{eq: first order convergence} and the energy bound show that the last three terms in \eqref{eq: unsigned identity} tend to zero in $L^1_{\mathrm{loc}}(\Omega)$.

Moreover, $|u|^2\to1$ in $L^1_{\mathrm{loc}}(\Omega)$ and $\chi$ is bounded, so $\int_1^{|u|^2}\chi(a)\,da\to0$ in $L^1_{\mathrm{loc}}(\Omega)$. Its Laplacian tends to zero against compactly supported smooth tests. By \eqref{eq: scalar mass}, $\alpha_\eps\,d\vol_g\to\mu$ distributionally. The local total variation bound $|q_\eps|\leq Ce_\eps$ and the vanishing negative part give convergence of the positive parts as Radon measures. Finally,
\begin{equation*}
\big||q_\eps|-(\alpha_\eps)_+\big|\leq|R_\eps|+(\alpha_\eps)_-,
\end{equation*}
which proves the second convergence. 

Let $P=g-\nu^\flat\otimes\nu^\flat$. The planar algebra gives
\begin{equation*}
|2S-|\nabla u|^2P|\leq C|\nabla u|\big(|\nabla u|^2-(*K)(\nu)\big)^{1/2}.
\end{equation*}
To see this, work in a local unitary trivialisation and an oriented orthonormal frame with $e_3=\nu$. Project $(\nabla^A_{e_1}u,\nabla^A_{e_2}u,\nabla^A_{e_3}u)$ onto the complex line $\{(z,iz,0):z\in\C\}$. The squared distance to this line is bounded by the defect in \eqref{eq: planar defect}, and on the line $2S=|\nabla^Au|^2P$. The change in the quadratic tensor is bounded by $C|\nabla^Au|$ times that distance.

Since $F^\top F=|F_\nabla|^2P$, the definition of $T_\eps$ gives
\begin{equation*}
T_\eps-e_\eps\nu^\flat\otimes\nu^\flat=|\nabla u|^2P-2S+\big(\eps^{-2}h^2-\eps^2|F_\nabla|^2\big)P.
\end{equation*}
Consequently,
\begin{equation*}
|T_\eps-e_\eps\nu^\flat\otimes\nu^\flat|\leq C|\nabla u|\big(|\nabla u|^2-(*K)(\nu)\big)^{1/2}+C\eps^{-1}|\xi_\eps|\big(h+\eps^2|F_\nabla|\big).
\end{equation*}
The first term tends to zero in $L^1(\Omega')$ by \eqref{eq: first order convergence} and the energy bound. For the second term, \eqref{eq: negative discrepancy} and Cauchy--Schwarz give an $L^1(\Omega')$ bound $C\eps$, since $\|h+\eps^2|F_\nabla|\|_{L^2}\leq C\eps$. The same algebra applies on the zero set with the chosen extension of $\nu$.
\end{proof}
\begin{proof}[Proof of \Cref{prop: integral curvature estimates}]
    By \eqref{eq: q decomposition},
\begin{equation*}
\nabla_{q_\eps}\nu=\alpha_\eps\nabla_\nu\nu+\nabla_{R_\eps}\nu.
\end{equation*}
On the core, $\eps^2|F_\nabla|\geq c>0$ and $e_\eps\leq C\eps^{-2}$. Since $|\alpha_\eps|\leq|q_\eps|\leq Ce_\eps$, we have $\eps^{-2}\alpha_\eps^2/|F_\nabla|^2\leq Ce_\eps$. Thus \eqref{eq: auxiliary integral curvature} and Cauchy--Schwarz give
\begin{equation}\label{eq: parallel turning}
\begin{split}
\int_{B_r(x)}|\alpha_\eps\nabla_\nu\nu|\,d\vol_g\leq{}&\left(\int_{\Omega'\cap\{|F_\nabla|>0\}}\eps^2|F_\nabla|^2|\nabla_\nu\nu|^2\,d\vol_g\right)^{1/2}\\
&\cdot\left(\int_{B_r(x)}\eps^{-2}\frac{\alpha_\eps^2}{|F_\nabla|^2}\,d\vol_g\right)^{1/2}\leq CE_\eps(u,\nabla;B_r(x))^{1/2}.
\end{split}
\end{equation}
The ratios are extended by zero off the core.

For the transverse contribution, \eqref{eq: transverse bound} gives on the core
\begin{equation*}
\begin{split}
|\nabla_{R_\eps}\nu|&\leq C|\nabla u||\nabla_\nu u||\nabla\nu|\\
&\leq C\eps^{-1}\big(|\nabla_\nu u|^2+\eps^2|F_\nabla||\nabla\nu|^2\big)\leq C\eps^{-1}h\mathcal S_\eps.
\end{split}
\end{equation*}
Here $h\geq c_0$ on the core. Choose a cutoff equal to one on $B_r(x)$, supported in $B_{2r}(x)$, with $|d\eta|\leq C/r$, and apply \eqref{eq: localized defect}. Combining it with \eqref{eq: parallel turning} proves \eqref{eq: turning estimate}. The local energy growth bound and the limit $\eps\to0$ at fixed $r$ give \eqref{eq: turning limit}.
\end{proof}

The only task left is to prove the main estimate.
\begin{proof}[Proof of \Cref{prop: curvature}]
    First we show that in every $\Omega'\Subset\Omega$,
\begin{equation}\label{eq: uniform defect}
\|\xi_\eps\|_{L^\infty(\Omega')}\leq C\eps^{-4/5}.
\end{equation}
In particular, for each fixed $c>0$ and all sufficiently small $\eps$,
\begin{equation}\label{eq: core lower bound}
\eps^2|F_\nabla|\geq c/2\quad\text{on }\Omega'\cap\{h\geq c\}.
\end{equation}

The smallness of $\eps\|\xi_\eps\|_{L^\infty}$ obtained here uses dimension three. Indeed, assuming analogous $L^2$ and Lipschitz bounds in dimension $n$, the interpolation argument below would give
\begin{equation*}
\eps\|\xi_\eps\|_{L^\infty}\leq C\eps^{(4-n)/(n+2)},
\end{equation*}
which yields smallness for $n<4$.

Choose $\Omega'\Subset\Omega''\Subset\Omega$. Equations \eqref{eq: pointwise} and \eqref{eq: curvature derivative} give $\operatorname{Lip}(\xi_\eps)\leq C\eps^{-2}$ on $\Omega''$, and $\|\xi_\eps\|_{L^\infty(\Omega'')}\leq C\eps^{-1}$. If $\|\xi_\eps\|_{L^\infty(\Omega')}>0$, a ball centred in $\overline{\Omega'}$, of radius comparable to $\eps^2\|\xi_\eps\|_{L^\infty(\Omega')}$, lies in $\Omega''$ and satisfies $|\xi_\eps|\geq\|\xi_\eps\|_{L^\infty(\Omega')}/2$. Consequently,
\begin{equation*}
c\eps^6\|\xi_\eps\|_{L^\infty(\Omega')}^5\leq\int_{\Omega''}\xi_\eps^2\,d\vol_g\leq C\eps^2.
\end{equation*}
This proves \eqref{eq: uniform defect}; \eqref{eq: core lower bound} follows from $\eps^2|F_\nabla|=h+\eps\xi_\eps$ and $\eps\|\xi_\eps\|_{L^\infty(\Omega')}\to0$. Next we prove \eqref{eq: auxiliary integral curvature}. Choose $\zeta\in C_c^\infty(\Omega)$ equal to one on a neighbourhood of $\overline{\Omega'}$. For the symmetric bilinear form $B$, set
\begin{equation*}
(\div B)(Y)=\sum_i(\nabla_{e_i}B)(e_i,Y).
\end{equation*}
The gauge expression for ${\bf W}_a$ is
\begin{equation*}
G[{\bf W}_a]=(\div B-\eps^{-2}dh)(X_a)+\sum_iB(e_i,\nabla_{e_i}X_a).
\end{equation*}
By \eqref{eq: projected fields}, the mixed terms vanish after summing the squares in $a$, so $\sum_aG[{\bf W}_a]^2\geq|\div B-\eps^{-2}dh|^2$. Moreover,
\begin{equation*}
G[\zeta{\bf W}_a]=\zeta G[{\bf W}_a]+B(X_a,\nabla\zeta).
\end{equation*}
Apply stability and \eqref{eq: gauge correction} to the compactly supported pairs $\zeta{\bf W}_a$, and use \eqref{eq: test bound}. Since $\eps^2|B|^2\leq Ce_\eps$, the cutoff term is bounded by the energy on $\spt\zeta$. It follows that
\begin{equation*}
\int_{\Omega'}\eps^2|\div B-\eps^{-2}dh|^2\,d\vol_g\leq C.
\end{equation*}
Since $dF_\nabla=0$, the vector field $|F_\nabla|\nu$ is divergence-free. On $\{|F_\nabla|>0\}$,
\begin{equation*}
\div B=d|F_\nabla|-|F_\nabla|(\nabla_\nu\nu)^\flat,\qquad \div B-\eps^{-2}dh=-|F_\nabla|(\nabla_\nu\nu)^\flat+\eps^{-1}d\xi_\eps.
\end{equation*}
Now use the interior bound for $d\xi_\eps$ from \Cref{thm: discrepancy} and conclude with \eqref{eq: auxiliary integral curvature}. It remains to prove \eqref{eq: localized defect}, using \eqref{eq: defect equation}. Indeed for every nonnegative $\phi\in C_c^\infty(\Omega)$,
\begin{equation}\label{eq: defect inequality}
\int_\Omega\phi\mathcal S_\eps\,d\vol_g\leq-\eps\int_\Omega\big(\langle d\xi_\eps,d\phi\rangle+\eps^{-2}|u|^2\xi_\eps\phi\big)\,d\vol_g+C\eps^2\int_\Omega\phi|F_\nabla|\,d\vol_g,
\end{equation}
where $C$ bounds the curvature on $\spt\phi$. To justify this inequality across $\{F_\nabla=0\}$, write $b=(*F_\nabla)^\sharp$ and regularise its norm by
\begin{equation*}
f_\delta=\sqrt{|b|^2+\delta^2},\qquad
\xi_{\eps,\delta}=\eps f_\delta-\eps^{-1}h.
\end{equation*}
The curvature equation and the one-form Weitzenb\"ock formula give the regularised defect
\begin{equation*}
\mathcal S_{\eps,\delta}
=|\nabla u|^2-\frac{\langle(*K)^\sharp,b\rangle}{f_\delta}
+\frac{\eps^2}{f_\delta}\big(|\nabla b|^2-|df_\delta|^2\big)\geq0.
\end{equation*}
The inequality follows from $|K|\leq|\nabla u|^2$ and $|df_\delta|\leq|\nabla b|$. The regularised version of \eqref{eq: defect inequality} has an additional error bounded by $\delta\int_\Omega\phi\,d\vol_g$. As $\delta\to0$, $\xi_{\eps,\delta}\to\xi_\eps$ in $H^1_{\mathrm{loc}}$ for fixed $\eps$. Off the zero set, $\mathcal S_{\eps,\delta}\to\mathcal S_\eps$, while on the zero set its lower limit is at least $|\nabla u|^2$. Fatou's lemma therefore proves \eqref{eq: defect inequality}.

Choose a fixed nonnegative $\phi\in C_c^\infty(\Omega)$ equal to one near $\overline{\Omega'}$. Integration by parts gives
\begin{equation*}
\begin{split}
\int_{\Omega'}\mathcal S_\eps\,d\vol_g
\leq{}&\eps\int_\Omega\xi_\eps\Delta\phi\,d\vol_g
-\eps^{-1}\int_\Omega|u|^2\xi_\eps\phi\,d\vol_g\\
&+C\eps^2\int_\Omega\phi|F_\nabla|\,d\vol_g\longrightarrow0.
\end{split}
\end{equation*}
Here \eqref{eq: strong scalar} controls the first two terms and \eqref{eq: scalar mass} controls the last. This proves \eqref{eq: full defect convergence}. The first-order defect is bounded by $\mathcal S_\eps$ off the zero set and by $2\mathcal S_\eps$ on it, proving \eqref{eq: first order convergence} for every measurable unit extension of $\nu$.

Finally, since $h\geq0$, test \eqref{eq: defect inequality} with $\eta h$, use the scalar equation and integrate by parts:
\begin{equation*}
\begin{split}
\int_\Omega\eta h\mathcal S_\eps\,d\vol_g\leq{}&-\eps\int_\Omega\eta\xi_\eps|\nabla u|^2\,d\vol_g+\eps\int_\Omega\langle\xi_\eps\,dh-h\,d\xi_\eps,d\eta\rangle\,d\vol_g\\
&+C\eps^2\int_\Omega\eta h|F_\nabla|\,d\vol_g.
\end{split}
\end{equation*}
By the pointwise and discrepancy estimates,
\begin{equation*}
\eps\left|\int_\Omega\eta\xi_\eps|\nabla u|^2\,d\vol_g\right|\leq\|\eps^{-1}\xi_\eps\|_{L^2(\Omega')}\left(\int_{\Omega_0}\eps^4|\nabla u|^4\,d\vol_g\right)^{1/2}\leq C\eps E_\eps(u,\nabla;\Omega_0)^{1/2}.
\end{equation*}
Also $\|h\|_{L^2(\Omega_0)}\leq C\eps E_\eps(u,\nabla;\Omega_0)^{1/2}$, $|dh|\leq|\nabla u|$, $\|d\xi_\eps\|_{L^2(\Omega')}\leq C$ and $\|\xi_\eps\|_{L^2(\Omega')}\leq C\eps$. The cutoff term is bounded by $C\eps^2\|d\eta\|_{L^\infty}E_\eps(u,\nabla;\Omega_0)^{1/2}$. Finally, $2h|F_\nabla|\leq e_\eps$, so the curvature term is at most $C\eps^2 E_\eps(u,\nabla;\Omega_0)$. Divide by $\eps$ to obtain \eqref{eq: localized defect}.
\end{proof}

\subsection*{Acknowledgments} MB was supported by the European Research Council under Grant Agreement No. 101165368 (MAGNETIC). AH was supported by the European Research Council under Grant Agreement No. 101076411 (ANGEVA). Views and opinions expressed are however those of the authors only and do not necessarily reflect those of the European Union or the European Research Council. Neither the European Union nor the granting authority can be held responsible for them.

\bibliography{Bib}
\bibliographystyle{siam}

\end{document}